\documentclass[letter,11pt]{article}
\usepackage[english]{babel}
\usepackage[utf8]{inputenc}

\usepackage[margin = 1.2 in, top = 1 in, bottom = 1.2 in]{geometry}
\usepackage{graphicx}
\usepackage{cancel}
\usepackage{float}
\usepackage{enumerate}
\usepackage{scalerel,stackengine}
\usepackage{graphics}
\usepackage{fancyhdr}
\usepackage{cancel}
\usepackage{enumerate}
\usepackage{amssymb}
\usepackage{amsmath}
\usepackage{hyperref}
\usepackage{mdwlist}
\usepackage{etoolbox}
\usepackage{latexsym}
\usepackage{amsthm}
\usepackage{multicol}
\usepackage{makeidx}

\usepackage{graphicx}
\usepackage{wrapfig}
\usepackage{dsfont}
\usepackage{mdframed,color}
\usepackage[dvipsnames]{xcolor}
\usepackage{mathtools}
\usepackage[capitalize]{cleveref}
\usepackage{tikz}
\usetikzlibrary{matrix}
\usetikzlibrary{positioning}

\newtheorem{theo}{Theorem}[section]
\newtheorem{cor}[theo]{Corollary}

\newtheorem{lemma}[theo]{Lemma}
\newtheorem{prop}[theo]{Proposition}
\newtheorem*{prop*}{Proposition}

\newtheoremstyle{definition}{2mm}{2mm}{}{}{\bfseries}{.}{.5em}{}
\theoremstyle{definition}

\newtheorem{remark}[theo]{Remark}
\newtheorem{ex}[theo]{Example}

\newcommand{\N}{\mathbb{N}}
\newcommand{\Z}{\mathbb{Z}}
\newcommand{\C}{\mathbb{C}}

\newcommand{\R}{\mathbb{R}}
\newcommand{\E}{\mathbb{E}}
\newcommand{\sF}{\mathcal{F}}

\newcommand{\sA}{\mathcal{A}}

\newcommand{\bE}{\mathop{\mathbb{E}}}

\newcommand{\norm}[1]{\left\lVert#1\right\rVert}
\newcommand{\ind}[1]{\mathbb{1}_{#1}}
\newcommand{\uclim}{\mathop{\textup{\textsf{UC-}}\lim}}
\newcommand{\gen}{\mathop{\textup{\textsf{gen}}}}

 \usepackage{stix} 
\usepackage{bm}
 \newcommand{\gec}{\text{\footnotesize~\rotatebox[origin=c]{270}{$\triangleplus$}~}}
\usepackage{titlesec}
\titleformat{\section}
  {\Large\center\bfseries}
  {\thesection.}{.7em}{}
\titlespacing*{\section}{0pt}{3.5ex plus 0ex minus 0ex}{1.5ex plus 0ex}
\titleformat{\subsection}
  {\center\bfseries}
  {\thesubsection.}{.7em}{}
\titlespacing*{\subsection}{0pt}{3.5ex plus 0ex minus 0ex}{1.5ex plus 0ex}
\titleformat{\subsubsection}
  {\center\bfseries}
  {\thesubsubsection.}{.7em}{}
\titlespacing*{\subsubsection}{0pt}{3.5ex plus 0ex minus 0ex}{1.5ex plus 0ex}

\addto\captionsenglish{}

\usepackage{titling}
\begin{document}
\title{Infinite linear patterns in sets of positive density}
\author{Felipe Hernández}
\date{\small \today}
{\bf \Large \maketitle}
\begin{abstract}
In this article we describe all possible infinite linear configurations that can be found in a shift of any set of positive upper Banach density. This simultaneously generalizes Szemerédi's theorem on arithmetic progressions and the recent density finite sums theorem of Kra, Moreira, Richter, and Robertson.
\end{abstract}

\section{Introduction}
Let $d\in \N$. Given sets $B_1,\ldots,B_d\subseteq \N$ and integers $w_1,\ldots,w_d\in \Z$, we define the ordered linear configuration
\begin{equation}\label{OLC}
    w_1B_1\gec \cdots \gec w_dB_d:= \{ w_1b_1+ \cdots +w_db_d \mid \text{ for each } i\in \{0,\ldots,d\},b_i\in B_i, \text{ and } b_1>\cdots>b_d\}.
\end{equation}

In \cite{kmrr25} Kra, Moreira, Richter, and Robertson established the following result.
\begin{theo}[{\cite[2025]{kmrr25}}]\label{conj-KMRR25}
If $A\subseteq \N$ has positive upper Banach density then for every $d\in \N$ there is an integer $t\geq 0$ and an infinite set $B \subseteq \N$ such that 
$$B ,B\gec B,\ldots, \underbrace{B\gec \cdots \gec B}_{d\text{ times}} \subseteq A-t. $$
\end{theo}
\cref{conj-KMRR25} is a follow-up to the main result in \cite{Kra_Moreira_Richter_Robertson:2023}, where the authors proved the case $d=2$ of \cref{conj-KMRR25}, resolving a long-standing conjecture of Erd{\H o}s.

In this paper we generalize \cref{conj-KMRR25} by replacing the sumset configurations with the more general ordered linear configurations defined in \cref{OLC}. This resolves \cite[Question 3.11]{kmrr2023problemsinfinitesumsetconfigurations}. 
\begin{theo}\label{Conjecture-1}
Let $d,r,k\in \N$. If $A\subseteq \N$ has positive upper Banach density, then there exist an integer $t\geq0$ and an infinite set $B\subseteq \N$ satisfying
$$kB\gec w_1B \gec \cdots \gec w_dB\subseteq A-t  $$
for all $w_1,\ldots,w_d \in \{-r,\ldots,r\}$ such that $k+\sum_{i=1}^j w_i\neq 0$ for all $j\in \{1,\ldots,d\}$.
\end{theo}
The case of \cref{Conjecture-1} regarding sumsets of the form $k B \gec w_1 B$ with $k,w_1\in \N$ was proved in \cite{kousek2025} using different techniques.  

For a set $B\subseteq \N$ and $d\in \N$ we write $B^{\otimes d}:= \{(b_1,\ldots,b_d): b_1>\cdots >b_d \in B \}$. For $j\in \{1,\ldots,d\}$ denote by $e_j$ the $j$-canonical vector of $\R^d$. We now restate \cref{Conjecture-1} in terms of linear forms, which makes the conditions over the admissible patterns more clear. 
\begin{theo}[Reformulation of \cref{Conjecture-1} using linear forms]\label{restatement-main-theorem}
    Let $m\in \N$ and let $\psi_1,\ldots,\psi_m:\Z^d\to \Z$ be non-constant linear forms such that 
    \begin{enumerate}
        \item\label{C1} $\psi_1(e_1)=\cdots= \psi_m(e_1) \in \N$, and
        \item\label{C2} for each $i\in\{1,\ldots,m\}$ and $j\in \{1,\ldots, d\}$, $\psi_i(e_1+\cdots+ e_j)\neq 0$.
    \end{enumerate}
    Then, for any $A\subseteq \N$ with positive upper Banach density there are an infinite set $B\subseteq \N$ and $t\in \N\cup\{0\}$ such that $\psi_1(B^{\otimes d}),\ldots, \psi_m(B^{\otimes d})\subseteq A-t$.
\end{theo}

Choosing $d=m\in \N$ and $\psi_1(x)=x_1$, $\psi_2(x)=x_1+x_2,\ldots, \psi_d(x)=x_1+\cdots+x_d$ we recover \cref{conj-KMRR25}. On the other hand, choosing $d=2$, $m\in \N$, and $\psi_1(x,y)=x$, $\psi_2(x,y)=x+y, \ldots, \psi_m(x,y)=x+(m-1)y$ we see that our result implies Szemerédi's theorem on arithmetic progressions \cite{Szemeredi75}. We emphasize that our methodology is similar to that of \cite{kmrr25}, and that we do not provide a new proof of Szemerédi's theorem, since we use it as a black box within our approach.

Conditions \ref{C1} and \ref{C2} in \cref{restatement-main-theorem} over the linear forms $\psi_1,\ldots,\psi_m:\Z^d\to \Z$ are not only sufficient, but also necessary. Condition \ref{C1} in \cref{restatement-main-theorem} was already highlighted in \cite[Example 2.3]{kmrr2023problemsinfinitesumsetconfigurations} where it is shown that there is a set $A\subseteq \N$ with positive upper 
 Banach density  such that for every $t\in \Z$ and infinite set $B\subseteq \N$, the set $B\gec B \cup 2B\gec B$ is not contained in $A-t$. This example can be slightly modified to construct, for any given pair of non-constant linear forms $\psi_1,\psi_2:\Z^d \to \Z$ with $\psi_1(e_1), \psi_2(e_2)\in \N$ distinct, a set $A$ with positive upper Banach density such that no shift of $A$ contains $\psi_1(B^{\otimes d} )$ and $\psi_2(B^{\otimes d})$ for any infinite set $B\subseteq \N$.
 \begin{ex}\label{ex1}
    Let $\psi_1,\psi_2:\Z^d\to \Z$ be non-constant linear forms such that $\psi_1(e_1)>\psi_2(e_1)>0$. Let $q=\psi_1(e_1)/\psi_2(e_1)$ and $c\in (1, q)$. Define $A=\N \cap \bigcup_{n\in \N }[ q^{2n},cq^{2n})$. Then, no shift of the set $A$ contains simultaneously the configurations $\psi_1(B^{\otimes d})$ and $\psi_2(B^{\otimes d})$ for an infinite set $B$. Indeed, assume by contradiction that $t\in \N_0$ is such that $\psi_1(B^{\otimes d}),\psi_2(B^{\otimes d})\subseteq A-t$ for an infinite set $B$. If $b_2,\ldots,b_d\in B$, then we can find $b\in B$ large enough such that for some $n\in \N$, $\psi_2(b,b_2,\ldots,b_d)+t$ lies in $[q^{2n},cq^{2n})$ with $\psi_1(0,b_2,\ldots,b_d)+t, \psi_2(0,b_2,\ldots, b_d)+t \in [-(q-c)q^{2n-1}/2,(q-c)q^{2n-1}/2]$. Then 
     $$ \psi_1(b,b_2,\ldots,b_d)+t \leq q\left(\psi_2(b,b_2,\ldots,b_d)+t\right) +(q-c)q^{2n-1}<q^{2(n+1)}
     $$
     and
     $$\psi_1(b,b_2,\ldots,b_d)+t\geq - (q-c)q^{2n} + q\left(\psi_2(b,b_2,\ldots,b_d)+t\right)\geq c q^{2n},$$
     which is a contradiction.
 \end{ex}

 Condition \ref{C2} in \cref{restatement-main-theorem} was already illustrated in \cite[cf. Theorem 2.20]{Beiglbock_Bergelson_Hindman_Strauss06} (see also \cite[Theorem 11.6]{Hindman1979}) where a construction from Ernst Straus is used to show that there is a set of positive upper Banach density $A$ such that no shift of $A$ contains $B-B$ for an infinite set $B\subseteq \N$. This example was also used in \cite{Kra_Moreira_Richter_Robertson:2023,kmrr2023problemsinfinitesumsetconfigurations,ackelsberg2024} to produce families of counterexamples concerning questions about sumsets in the setting of more general abelian groups. We use the same example to prove that for any non-constant linear form $\psi:\Z^d\to \Z$ with $\psi(e_1+\cdots+e_j)=0$ for some $j\in \{1,\ldots,d\}$, there is a set $A\subseteq \N$ with positive upper Banach density such that no shift of $A$ contains $\psi(B^{\otimes d})$ for an infinite set $B\subseteq \N$.
\begin{ex}\label{Straus-Example}
    Let $\eta\in (0,1)$. Consider a sequence $(k_n)_{n\in \N}$ of natural numbers such that $\sum_{n\in \N} k_n^{-1}<\eta$, and define $A=\N \setminus \left( \bigcup_{n\in \N} k_n\N + n\right)$. Then the upper Banach density of $A$ is at least $1-\eta$. For the sake of contradiction, assume there are a non-constant linear form $\psi:\Z^{d}\to \Z$ and $j\in \{1,\ldots,d\}$ such that $\psi(e_1+\cdots +e_j)=0$, and there is an increasing sequence $B=(b_n)_{n\in \N}\subseteq \N$ satisfying $\psi(B^{\otimes d})\subseteq A-t$ for some $t\in \N$. Without loss of generality $j=d$, as we can replace $t$ by $t':=t+\psi(0,\ldots,0,b_{d-(j-1)},\ldots,b_1)$, yielding for $B'=(b_{i})_{i\geq j}$ that $\psi'(B'^{\otimes j })\subseteq A-t',$ where $\psi'(x_1,\ldots,x_{j})= \psi(x_1,\ldots,x_{j},0, \ldots,0) $.
    By pigeonhole principle, we know that there is an infinite subset $C\subseteq B$ and $l\in \{0,\ldots, k_t\}$ such that $c\equiv l \text{ mod }k_t$ for each $c\in C$. Thus, for any $c_1>\cdots >c_d $ in $C$, $ \psi(c_1,\ldots,c_d)+t\in k_t\Z +t. $ Moreover, if we take $c_1$ large enough, we can ensure $\psi(c_1,\ldots,c_d)+t\in k_t\N +t$,
    implying that $A$ intersects $k_t\N+ t$, which is a contradiction.
\end{ex}

 It is also worth mentioning that it is not possible to remove the restriction $b_1>\cdots >b_d$ in \cref{Conjecture-1}, see \cite[Example 2.3]{kmrr2023problemsinfinitesumsetconfigurations}. Variants of these problems without this restriction were studied in \cite{kr24} for $\N$ and in \cite{ckmr25} for general abelian groups.

The paper is organized as follows. \cref{sec2} contains all basic notions from ergodic theory that we will use. \cref{sec3} focuses on defining a special measure playing a central role in our proof of \cref{Conjecture-1}. The key property of this measure is captured in \cref{Lemma-new-method}, which is the main technical result of this paper. \cref{sec4} is dedicated to proving \cref{Conjecture-1} assuming \cref{Lemma-new-method}. Finally, \cref{sec5} is devoted to demonstrating \cref{Lemma-new-method}. This is done by reducing the study of the ergodic average in \cref{Lemma-new-method} in the original system to the study of the respective ergodic average in the maximal pronilfactor of the system. The fact that nilsystems are uniquely ergodic allows to rewrite the aforementioned ergodic average, allowing the use of a uniform version of Szemerédi's theorem (see \cref{Uniform-Szemeredi}) to conclude the proof.

\subsection*{Acknowledgments}
The author is grateful to Florian Richter for his valuable guidance and numerous helpful comments on this article. We also thank Ethan Ackelsberg for fruitful discussions and for pointing out \cref{Straus-Example}, and Bryna Kra for many comments on an early version of the article. We thank the anonymous referee for helpful suggestions and comments.

\section{Preliminaries}\label{sec2}

\subsection{Notation}
For $s\in \N$ we denote $[s]:=\{1,\ldots,s\}$. For a sequence $(x_n)_{n\in \N}$ we define its \textit{uniform Cesàro limit} as   
$$\uclim_{n\to \infty} x_n :=\lim_{N-M\to \infty} \frac{1}{N-M}\sum_{n=M+1}^N x_n$$
whenever this limit exists. A F\o lner sequence $(\Phi_N)_{N\in\N}$ in $\N$ is a sequence of subsets of $\N$ such that 
$$ \lim_{N\to \infty} \frac{|\Phi_N \Delta (\Phi_N+k)|}{|\Phi_N|}=0 \text{ for all }k\in \N.$$
We say that a set $A\subseteq \N$ has positive upper Banach density if there is a F\o lner sequence $(\Phi_N)_{N\in \N}$ such that 
$$ \lim_{N\to \infty} \frac{|A \cap \Phi_N|}{|\Phi_N|}>0.$$
For the rest of the article, we fix $k\in \N$. We will use the alphabet $\sA=\{-r,\ldots,-1,1 \ldots,r\}$ and we will denote all the words of length $d$ or less\footnote{The set $\mathcal{A}^0$ will denote the singleton that contains the empty word $\varepsilon$.} in this alphabet as $\sA^{\leq d}= \sA^0 \cup \sA^1\cup\cdots\cup \sA^{d}$. We define the set of forbidden words $\mathcal{F}$ in $\mathcal{A}^{\leq d}$ as the set of $w\in \sA^{\leq d}$ such that for some $j\in \{1,\ldots,d\}$, $k+\sum_{i=1}^j w_i=0$. These words represent the forbidden configurations associated to \cref{Straus-Example}. We also define our set of admissible words $\mathcal{W}_d$ as $\sA^{\leq d}\setminus \sF$. In addition, we denote the set of nonempty words in $\mathcal{W}_d$ as $\mathcal{W}_d^{+}=\mathcal{W}_d\setminus\{\varepsilon\}$. For each $w\in \mathcal{W}_d^{+}$ we denote by $w^-$ the word $w$ without its last letter.

We define the order $\prec$ of the elements in $\mathcal{W}_d$ by increasing length and lexicographic order $\leq_{\text{lex}}$ for words of the same length, i.e. for $w,w'\in \mathcal{W}_d$ we have
\begin{equation*}
    w\prec w' \iff (|w|<|w'|) \lor (|w|=|w'| \land w\leq_{\text{lex}} w'). 
\end{equation*}

Let $S$ be a set. For a point $s\in S$ and a finite set $\Gamma\subseteq \N$ we will denote $s^{\Gamma}=(s,\ldots,s) $ the vector of length $|\Gamma|$ with $s$ in each coordinate. Likewise, for a transformation $F:S\to S$ over $S$, we will denote $F^{\Gamma}:S^{|\Gamma|}\to S^{|\Gamma|}$ the application of $F$ coordinatewise to each $\Vec{s}\in S^{|\Gamma|}$.
\subsection{Basic notions in Ergodic Theory}
Let $(X,T)$ be a \textit{topological dynamical system}, this is, a compact metric space $X$ with a homeomorphism $T:X\to X$ thereof. If, in addition, we have a $T$-invariant Borel probability measure $\mu$ on $X$, then—following the notation from \cite{Kra_Moreira_Richter_Robertson:2022}—we abuse language and refer to the triple $(X, \mu, T)$ as a \textit{system}. We say that $(X,\mu,T)$ is \textit{ergodic} if any $T$-invariant Borel set has measure either $0$ or $1$. A point $y\in X$ is called \textit{generic for $\mu$ along a F\o lner sequence $\Phi$}, which we denote by $y\in \gen_T(\mu,\Phi)$, if for each continuous function $f\in C(X)$
$$\int fd\mu =\lim_{N\to \infty} \frac{1}{|\Phi_N|} \sum_{n\in \Phi_N} f(T^ny). $$

A system $(Y,\nu,S)$ is a measurable factor of a system $(X,\mu,T)$ if there is a measurable map $\pi:X\to Y$ such that $\pi\mu=\nu$ and $S\circ \pi(x)= \pi\circ T(x)$ for $\mu$-a.e. $x\in X$. 
For a factor map $\pi:(X,\mu,T)\to (Y,\nu,S)$ there are measures $(\mu_y)_{y\in Y}$ defined $\nu$-a.e. on $X$ called the \textit{desintegration of $\mu$ with respect $\pi$} (see for example \cite[Theorem 5.14]{Einsiedler2011}) such that for each bounded and measurable function $f:X\to \C$ the map
    $$y\in Y \mapsto \int f d\mu_y $$
    is $\nu$-a.e. defined and Borel measurable, satisfying that all $V\in \mathcal{B}(Y)$, 
     $$\int_{\pi^{-1}(V)} f d\mu= \int_V \left( \int f d\mu_y\right) d\nu(y),$$
    and that for $\nu$-a.e. $y\in Y$, $\mu_{y}(\pi^{-1}(\{y\}))=1$. 

Let $\{(X_i,\mu_i,T_i)\}_{i\in \N}$ be a sequence of systems with continuous factor maps $\pi_i:X_i\to X_{i-1}$ between them. The \textit{inverse limit} $(X,\mu,T)$ of $\{(X_i,\mu_i,T_i)\}_{i\in \N}$ is the unique system such that there exist continuous factor maps $\psi_i:X\to X_i$ such that $\pi_i\circ \psi_i= \psi_{i-1}$ for each $i>1$, and such that 
$$C(X)=\overline{\bigcup_{i\in \N} C(X_i)\circ \psi_i}. $$

From this section onwards we will need the theory of pronilfactors. As we will only briefly recall such theory, we advise any interested reader to consult \cite{Host_Kra_nilpotent_structures_ergodic_theory:2018} for further details. An $s$\textit{-step nilmanifold} 
$X=G/\Gamma$ is a compact nilmanifold where $G$ is an $s$-step nilpotent Lie group and $\Gamma$ is a discrete cocompact subgroup of $G$. The group $G$ acts on $X$ by left translation and 
$\mu$ denotes the \textit{Haar measure} on $X$, this is, the unique Borel probability measure on $X$, which is invariant under this action. For a fixed element $\tau\in G$, we define the transformation $T:X\to X$ by left translation of $\tau$. Then, the triple $(X=G/\Gamma,\mu,T)$ constitutes an $s$\textit{-step nilsystem}.
An inverse limit of s-step nilsystems is called an s-step pronilsystem.

For a system $(X,\mu,T)$, we define the uniformity seminorms introduced in the ergodic case in \cite{Host_Kra_nonconventional_averages_nilmanifolds:2005} and then in the nonergodic case in \cite{CFH11}. Let $f\in L^\infty(X)$. For $s=0$
$$\norm{f}_{U^0(X,T,\mu)}:=\int_X f d\mu.$$
For $s\geq 0$ we define 
\begin{equation}\label{eq-US}
   \norm{f}_{U^{s+1}(X,\mu,T)}:=\lim_{H\to \infty} \left( \frac{1}{H}\sum_{h=0}^{H-1} \norm{T^h f \cdot \overline{f}}_{U^{s}(X,\mu,T)}^{2^s} \right)^{1/2^{s+1}}. 
\end{equation}
In \cite{Host_Kra_nonconventional_averages_nilmanifolds:2005} is proved that the limit in \cref{eq-US} always exists and that for $s\geq 1$, $\norm{\cdot}_{U^s(X,\mu,T)}$ defines a seminorm on $L^\infty(\mu)$. If in addition $(X,\mu,T)$ is ergodic, then the main result in \cite{Host_Kra_nonconventional_averages_nilmanifolds:2005} states that there is a factor $(Z,m,T)$, which is the maximal factor that is isomorphic to an s-step pronilsystem, such that for any $f\in L^\infty(X)$ we have 
$$\norm{f}_{U^{s+1}(X,\mu,T)}=0 \iff \E(f |Z)=0. $$
The factor $(Z,m,T)$ will be called the \textit{$s$-step pronilfactor of} $(X,\mu,T)$.

\section{The progressive measure}\label{sec3}

For a system $(X,\mu,T)$, we define the transformations
\begin{equation*}
    \tilde{T}:=  T^k \times(T^{w_{|w|}})_{w\in \mathcal{W}_d^{+}} ,\text{ and } \Vec{T}:= (T^{k+\sum_{i=1}^{|w|}w_i})_{w\in \mathcal{W}_d}
\end{equation*}
on $X^{|\mathcal{W}_d|}$. These transformations are the analogous of the transformations $T\times T$ and $T\times T^2$ presented in \cite{Kra_Moreira_Richter_Robertson:2023}, and of the transformations $T\times\cdots\times T$ and $T\times T^2\times\cdots\times T^d$ in \cite{kmrr25}. 

Now, we construct an analog to the progressive measure defined in \cite{kmrr25}, adapted to our setting. From now on, we consider an ergodic system $(X,\mu,T)$ with continuous factor map $\pi:X\to Z$ to its $L:=|\mathcal{W}_d^{+}|$-step pronilfactor $(Z,m,T)$, and we fix a point $a\in X$ such that $\mu(\overline{ \{T^n a : n\in \N\}})=1$. Let $z\to \eta_z$ be the desintegration of $\mu$ with respect to $\pi$. By \cite[Theorem 11, Chapter 11]{Host_Kra_nilpotent_structures_ergodic_theory:2018} the orbit closure of $\pi(a)^{\mathcal{W}_d}$ under $\Vec{T}$ is a uniquely ergodic nilsystem. We define a measure $\tilde{\sigma}$ on $Z^{|\mathcal{W}_d|}$ and a corresponding lift $\sigma$ to $X^{|\mathcal{W}_d|}$ as follows:
        \begin{equation*}
            \tilde{\sigma}=\uclim_{n\to\infty} \Vec{T}^n \delta_{\pi(a)^{\mathcal{W}_d}} \text{ and }   \sigma=\int_{Z^{|\mathcal{W}_d|}} \prod_{w\in \mathcal{W}_d} \eta_{u_w}  d\tilde{\sigma}(u).
        \end{equation*}     
By a similar argument as in \cite[Lemma 4.6]{kmrr25}, for each $w\in \mathcal{W}_d$, there is $c_w>0$ such that $\sigma_w\leq c_w \mu$.
        
  We also define on $Z$ the measure $$m_k=\uclim_{n\to \infty} T^{kn}\delta_{\pi(a)},$$ which is a measure supported on $Z_k:=\overline{\mathcal{O}_{T^k}(\pi(a))}$, and we define its lift to $X$ defined by $$\mu_k=\int_Z \eta_u dm_k(u).$$
We now state some properties of these measures:
    \begin{prop}\label{ergodic-decom-Tk}
    Let $(X,\mu,T)$ be an ergodic system, and let $(Z,m,T)$ and $(Z_k,m_k,T^k)$ be defined as above. Then, the ergodic decomposition of $\mu$ with respect $T^k$ is given by $\mu= \frac{1}{k}\sum_{j=0}^{k-1}  T^j\mu_k.$ In addition,  there exists $\ell|k$ such that $ \norm{\E(g | Z)}_{L^1(m)}=  \ell  \norm{\E( g | Z_k)}_{L^1(m_k)}.$
    \end{prop}
    \begin{proof}
     By unique ergodicity we have that 
    \begin{equation*}
        m= \uclim_{n\to \infty} T^n\delta_{\pi(a)} = \frac{1}{k}\sum_{j=0}^{k-1}  \uclim_{n\to \infty }T^{kn+j}\delta_{\pi(a)}= \frac{1}{k}\sum_{j=0}^{k-1} T^j  \uclim_{n\to \infty }T^{kn}\delta_{\pi(a)}= \frac{1}{k}\sum_{j=0}^{k-1} T^jm_k.
    \end{equation*}
    Consequently
    \begin{equation*}
        \mu= \int_Z \eta_u dm(u)=  \frac{1}{k}\sum_{j=0}^{k-1}  \int_Z \eta_u dT^jm_k(u)=  \frac{1}{k}\sum_{j=0}^{k-1}  T^j\int_Z \eta_u dm_k(u)=  \frac{1}{k}\sum_{j=0}^{k-1}  T^j\mu_k.
    \end{equation*}
    For the second part, notice that $\{T^iZ_k\}_{i=0}^{k}$ are either disjoint or equal by distality of $(Z,m,T^k)$. Defining $\ell\in \N$ as the minimum natural number such that $T^\ell Z_k=Z_k$, we have that $\ell|k$ given that $\ell$ is minimum period and $k$ is a period, and 
       \begin{equation*}
           T^{i\ell+j}Z_k=T^{j}Z_k, \text{ and } Z=\bigsqcup_{j=0}^{\ell-1}T^jZ_k.
       \end{equation*}
      Finally, we notice that
    \begin{equation*}
       \E(g | Z)(z)= \int_{Z} g d\eta_z= \sum_{j=0}^{\ell-1} \int_{T^{-j}Z_k} g d\eta_z= \sum_{j=0}^{\ell-1} \int_{Z_k} T^jg d\eta_{z} = \sum_{j=0}^{\ell-1} \E( T^jg | Z_k)(z).
    \end{equation*}  
    In consequence,
      \begin{align*}
           \norm{\E(g | Z)}_{L^1(m)} 
           &= \int_{Z} |\sum_{j=0}^{\ell-1} T^j\E( g | Z_k)(z) | dm(z)       
           \\&=  \sum_{j=0}^{\ell-1}\int_{Z} |\E( g | Z_k)(z) | dT^jm_k(z)=  \ell\norm{\E( g | Z_k)(z)}_{L^1(m_k)}.
      \end{align*}
    \end{proof}

The key ingredient for proving the main theorem is the following result, which illustrates that $\sigma$ has the analogous properties to the measure defined in \cite[Definition 3.1]{kmrr25}.
\begin{lemma}\label{Lemma-new-method}
    Let $\Phi$ be a F\o lner sequence with $a\in gen_{T^k}(\mu_k,\Phi)$. For all non-negative continuous function $F\in C(X^{|\mathcal{W}_d|})$ we have that
    \begin{equation*}
        \int F d\sigma >0 \Rightarrow
   \liminf_{N\to\infty} \bE_{n\in \Phi_N}\int    \tilde{T}^n F\left(a, (x_{w^-})_{w\in \mathcal{W}_{d}^{+}} \right)\cdot F(x)  d\sigma(x) >0.
    \end{equation*}
\end{lemma}
\begin{remark}\label{useful-remark}
    In the case $F=\bigotimes_{w\in \mathcal{W}_d} f_w$, \cref{Lemma-new-method} says that
    $$\int \bigotimes_{w\in \mathcal{W}_d} f_w d\sigma >0 \Rightarrow    \liminf_{N\to\infty} \bE_{n\in \Phi_N}T^{kn} f_\varepsilon(a) \int  \Big( (\bigotimes_{w\in \mathcal{W}_{d}\setminus\sA^d}  \prod_{\substack{j\in\sA\\wj\notin \sF}} T^{jn}f_{wj} )\otimes 1^{\sA^d\setminus \sF} \Big) \cdot \bigotimes_{w\in \mathcal{W}_d}f_{w}d\sigma>0.  $$
    In addition, \cref{Lemma-new-method} remains valid when $f_w=\ind{U_w}$ with $U_w\subseteq X$ an open set for each $w\in \mathcal{W}_d$, by an approximation by below argument. In consequence, if
    $$\sigma\left(\prod_{w\in \mathcal{W}_d} U_w\right)>0, $$
    then $$\liminf_{N\to\infty} \bE_{n\in \Phi_N} \ind{U_\varepsilon}(T^{kn}a)\cdot  \sigma\left(\Bigg( \Big(\prod_{w\in \mathcal{W}_{d}\setminus\sA^d} \bigcap_{\substack{j\in \sA \\ wj\notin \sF}}T^{-jn} U_{wj} \Big)\times X^{|\mathcal{A}^d\setminus \sF|} \Bigg)\cap \prod_{w\in \mathcal{W}_d} U_w  \right)>0. $$
\end{remark}
We will need the following uniform version of Szemerédi's theorem, which is a direct consequence of \cite[Theorem A.2]{BKLR19} and Furstenberg correspondence principle.
\begin{theo}[Uniform version of Szemerédi's theorem]\label{Uniform-Szemeredi}
Let $l\in \N$. For every $\eta>0$, there is $c>0$ such that for every set $E\subseteq \N$ with $\overline{d}(E)>\eta$ one has
$$\overline{d}\left(\left\{m\in \N \mid \overline{d}\Big(E\cap (E-m)\cap \cdots\cap (E-lm)\Big)>c \right\}\right)>c. $$
\end{theo}
We remark that \cref{Uniform-Szemeredi} can also be obtained using \cite[Theorem F2 and Theorem 2.1]{MR1784213} and Furstenberg correspondence principle.

We will apply \cref{Lemma-new-method} iteratively to obtain the desired linear patterns in \cref{Conjecture-1}. To initialize this process we need the following proposition which is a consequence of \cref{Uniform-Szemeredi}.
\begin{prop}\label{positive-sigma-measure}
     Let $\Phi$ be F\o lner sequence with $a\in gen_{T^k}(\mu_k,\Phi)$. For each non-negative continuous function $f\in C(X)$ we have that
      \begin{equation*}
        \int f d\mu >0 \Rightarrow
   \liminf_{N\to\infty} \bE_{n\in \Phi_N}\int  T^{n}f \otimes \cdots \otimes T^n f  d\sigma >0.
    \end{equation*}
\end{prop}
\begin{remark}\label{useful-remark-ii}
As well as \cref{Lemma-new-method}, \cref{positive-sigma-measure} remains valid if $f=\ind{U}$ where $U$ is an open subset of $X$.
\end{remark}
The proof of \cref{positive-sigma-measure} is a good outline of the ideas behind \cref{Lemma-new-method}, so we present it immediately. 
\begin{proof}
  By definition of $\sigma$ we have
\begin{equation*}
  \int (T\times\ldots \times T)^t\bigotimes_{w\in \mathcal{W}_d} f d\sigma= \int (T\times\ldots \times T)^t\bigotimes_{w\in \mathcal{W}_d} \E(f \mid Z) d\tilde{\sigma}. 
\end{equation*}
In consequence, it is enough to prove that for every $\eta>0$ there is a $\delta>0$ such that for every $f\in L^\infty(Z)$ non-negative and bounded by $1$ function we have 
\begin{equation}\label{eq-23}
   \int f dm > \eta \Rightarrow \uclim_{t} \int (T\times\ldots \times T)^t\bigotimes_{w\in \mathcal{W}_d} f d\tilde{\sigma}>\delta.
\end{equation}
Since this statement gives a uniform $\delta>0$ given $\eta>0$, it is enough to prove it for continuous functions. Let $f\in C(Z)$ non-negative and bounded by $1$ function such that
\begin{equation*}
   \uclim_{t\to\infty} f(T^ta) = \int f dm > \eta.
\end{equation*}
That being so, we have that the set $
    P=\{t\in \N \mid f(T^ta)>\eta/2\}$ has lower Banach density greater than $\eta/(2-\eta)$. As a result, by Szemerédi's theorem (\cref{Uniform-Szemeredi}), there is $c>0$ such that
\begin{equation*}
    \overline{d}\Big(\left\{ m\in \N\mid  \overline{d}\Big( P+(2rd)m\cap\cdots P+m\cap P \cap P-m \cdots \cap P-(2rd)m \Big)      >c    \right\}  \Big)>c
\end{equation*}
For $t,m\in \N$ with $t\in  P+(2rd)m\cap\cdots P+m\cap P \cap P-m \cdots \cap P-(2rd)m$ we have that $t+im\in P$ for all $i\in \{-2rd,\ldots,2rd\}$. In consequence, $f(T^{t+im}a)>\eta/2$ for all $i\in  \{-2rd,\ldots,2rd\}$ and such $t,m\in \N$. Therefore, using the uniquely ergodicity of the orbit of $a^{\mathcal{W}_d}$ through the transformations $\Vec{T}$ and $T\times\cdots\times T$ given by \cite[Theorem 17, Chapter 11]{Host_Kra_nilpotent_structures_ergodic_theory:2018}, we conclude that
\begin{equation*}
    \uclim_{t\to \infty}  \int T^tf\otimes\cdots \otimes T^tf d\tilde{\sigma}= \uclim_{m,t\to\infty} \prod_{w\in \mathcal{W}_d} f(T^{t+m(\sum_{i=1}^{|w|} w_i )}a)\geq (\eta/2)^{|\mathcal{W}_d|} c^2.
\end{equation*}
Defining $\delta:=  (\eta/2)^{|\mathcal{W}_d|} c^2$ we conclude \cref{eq-23}. 
\end{proof}

\section{The main result}\label{sec4}
Similarly to the proof of \cref{conj-KMRR25} in \cite{kmrr25}, we make use of a version of the Furstenberg correspondence principle to reduce the combinatorial statement to an ergodic theoretical framework. Specifically, we utilize the following slight modification of Theorem 2.10 from \cite{Kra_Moreira_Richter_Robertson:2022}. 
\begin{prop}\label{theo2.10}
Given a set $A \subseteq \N$ with positive upper Banach density, there exists an ergodic system $(X,\mu,T)$, a F\o lner sequence $\Phi$, a point $a\in gen_{T^k}(\mu_k,\Phi)$ and $\mu(\overline{\{T^na :n\in \N \}})=1$,
and a clopen set $E\subseteq X$ such that $\mu(E) > 0$ and $A = \{n \in \N \mid T^n a \in  E\}$.
\end{prop}
\begin{proof}
    The proof is identical as in \cite{Kra_Moreira_Richter_Robertson:2022}, with the exception that when invoking \cite[Proposition
3.9]{Furstenberg1981} at the very end of the proof, we do it for $(X,\mu_k,T^k)$ instead of $(X,\mu,T)$.
\end{proof}
Now, we outline the proof of \cref{Conjecture-1} assuming \cref{Lemma-new-method}.
\begin{proof}[{Proof of \cref{Conjecture-1}}]
By \cref{theo2.10} there is an ergodic system $(X,\mu,T)$, a point $a\in gen_{T^k}(\mu_k,\Phi)$ with $\mu(\overline{\{T^na :n\in \N \}})=1$, and a clopen set $E\subseteq X$ such that $\mu(E) > 0$ and $A = \{n \in \N \mid T^n a \in  E\}$. By \cite[Proposition 5.7 and Lemma 5.8]{Kra_Moreira_Richter_Robertson:2022} we can assume without loss of generality that $(X,\mu,T)$ has continuous factor map to its $L$-step pronilfactor. 

From \cref{positive-sigma-measure} and \cref{useful-remark-ii} we deduce that there is $t\in \N$ such that $\sigma(T^{-t}E\times \cdots \times T^{-t}E)>0$. We define $B$ inductively. First, by \cref{Lemma-new-method} and \cref{useful-remark} with the function $\ind{E} \otimes\cdots \otimes \ind{E}$ we get that there is $b_1\in \N$ such that $T^{kb_1}a\in E$ and 
$$\sigma\left(\Bigg( \Big(\prod_{w\in \mathcal{W}_{d}\setminus\sA^d} \bigcap_{\substack{j\in \sA \\ wj\notin \sF}}T^{-jb_1} E \Big)\times X^{|\mathcal{A}^d\setminus \sF|} \Bigg)\cap \prod_{w\in \mathcal{W}_d} E  \right)>0.
$$
Suppose we have found an increasing sequence $b(1)<b(2)<\cdots <b(n)$ satisfying that for all $w\in \mathcal{W}_d$ and $n\geq i_0>i_1>\cdots >i_{|w|}\geq 1 $, $kb(i_0) + w_1 b(i_1)+\cdots + w_{|w|}b(i_{|w|}) \in A$ and 
\begin{equation}\label{set-induction}
\sigma\left(\Bigg( \Big(  \prod_{w\in \mathcal{W}_{d}\setminus\sA^d}  \bigcap_{\substack{v\in \mathcal{A}^{\leq d-|w|}\\ wv\notin \sF}}\bigcap_{n\geq i_{1}>\cdots >i_{|v|}\geq 1} T^{-(v_1b(i_{1})+\cdots +v_{|v|}b(i_{|v|}))}E \Big) \times X^{|\mathcal{A}^d\setminus\sF|}\Bigg)\cap  \prod_{w\in \mathcal{W}_d} E  \right)>0 .   
\end{equation}
Using \cref{Lemma-new-method} and \cref{useful-remark} with the indicator function of the set in \cref{set-induction} we can find $b(n+1)>b(n)$ such that for each $w\in \mathcal{W}_d$ and $n\geq i_{1}>\cdots >i_{|w|}\geq 1$, $kb(n+1)+w_1b(i_{1})+\cdots w_{|w|}b(i_{|w|}))\in \N$; and
$$T^{kb(n+1)}a\in \bigcap_{w\in \mathcal{W}_d}\bigcap_{n\geq i_{1}>\cdots >i_{|w|}\geq 1} T^{-(w_1b(i_{1})+\cdots w_{|w|}b(i_{|w|}))}E. $$
 Equivalently, we have  that for all $w\in \mathcal{W}_d$ and $n+1\geq i_0>i_1>\cdots >i_{|w|}\geq 1 $, $kb(i_0) + w_1 b(i_1)+\cdots + w_{|w|}b(i_{|w|}) \in A$. By \cref{Lemma-new-method} we also have that 
 \begin{align*}
     0<& \sigma\Bigg( \Big(\Big(  \prod_{w\in \mathcal{W}_{d}\setminus\sA^d}  \bigcap_{\substack{j\in \sA\\ wj\notin\sF}} \bigcap_{\substack{v\in \mathcal{A}^{\leq d-|w|-1}\\ wjv\notin \sF}}\bigcap_{n\geq i_{1}>\cdots >i_{|v|}\geq 1} T^{-(jb(n+1)+v_1b(i_{1})+\cdots v_{|v|}b(i_{|v|}))}E\Big)\times X^{|\mathcal{A}^d\setminus\sF|} \Big) \\
     &\cap \Big( \Big(  \prod_{w\in \mathcal{W}_{d}\setminus\sA^d}  \bigcap_{\substack{v\in \mathcal{A}^{\leq d-|w|}\\ wv\notin\sF}}\bigcap_{n\geq i_{1}>\cdots >i_{|v|}\geq 1} T^{-(v_1b(i_{1})+\cdots +v_{|v|}b(i_{|v|}))}E\Big) \times X^{|\mathcal{A}^d\setminus\sF|} \Big)\cap \prod_{w\in \mathcal{W}_d} \Bigg)\\
     &= \sigma\left(\Bigg(\Big(  \prod_{w\in \mathcal{W}_{d}\setminus\sA^d}  \bigcap_{\substack{v\in \mathcal{A}^{\leq d-|w|}\\ wv\notin\sF}}\bigcap_{n+1\geq i_{1}>\cdots >i_{|v|}\geq 1} T^{-(v_1b(i_{1})+\cdots +v_{|v|}b(i_{|v|}))}E \Big)\times X^{|\mathcal{A}^d\setminus\sF|} \Bigg)\cap  \prod_{w\in \mathcal{W}_d} E  \right)
 \end{align*}
which concludes the induction. We have shown that there is an infinite set $B\subseteq \N$ such that 
\begin{equation*}
   \bigcup_{w\in \mathcal{W}_d}  kB\gec w_1B \gec \cdots \gec w_{|w|}B\subseteq A-t,
\end{equation*}
concluding.
\end{proof}

\section{Proving \cref{Lemma-new-method}}\label{sec5}
Now we prove \cref{Lemma-new-method}. First, we prepare the main tools for reducing the problem to study the $|\mathcal{W}_d^{+} |$-pronilfactor of $(X,\mu,T)$. For doing such, we will need some properties of structure theory. The following proposition is already known in the literature but we were unable to find a proper statement of it in its general form. 
\begin{prop}\label{W-erg-avera}
    Let $(X,\mu,T)$ be a measure preserving system. Let $(b_n)_{n\in \N}$ be a bounded complex sequence bounded by $1$. For any integer $s\geq 1$ and integers $c_1,\ldots,c_s\geq 1$, there is a constant $C$, independent of the system, with the following property: for all $f_1,\ldots,f_s\in L^\infty(\mu)$ with $\norm{f_i}_{L^\infty(\mu)}\leq 1$ and all F\o lner sequences $(\Phi_N)_{N\in \N}$, we have 
    \begin{equation}\label{eq-prop5.1}
      \limsup _{N \rightarrow \infty}\left\|\frac{1}{\left|\Phi_N\right|} \sum_{n \in \Phi_N} b_n T^{c_1 n} f_1 \cdots T^{c_s n} f_s\right\|_{L^2(\mu)} \leq C\min \left\{\left\|f_i\right\|_{U^{s+1}(X, \mu, T)}: i\in [s]\right\}.  
    \end{equation}
   $$ $$
\end{prop}
The case of \cref{W-erg-avera} where $c_i=i$ for each $i\in [s]$ can be found in \cite[Corollary 7.3]{MR2544760}.
\begin{proof}
By \cite[Lemma 6.10]{kmrr25} we have that the left-hand side of \cref{eq-prop5.1} is bounded by
$$
 \limsup _{N \rightarrow \infty}   \norm{\frac{1}{|\Phi_N|}\sum_{n\in \Phi_N}  (T\times T)^{c_1 n} f_1\otimes \overline{f}_1\cdots (T\times T) ^{c_s n} f_s\otimes \overline{f}_s)}_{L^2(\mu\times \mu)}^{1/2}.
$$
On the other hand, by \cite[Lemma 6.7, part (iii)]{kmrr25} there is $C>0$ such that the last expression is bounded by 
$$ C \min\{ \norm{f_i\otimes \overline{f}_i}_{U^s(X\times X,\mu\times \mu,T\times T)}^{1/2} : i\in [s]\} \leq  C \min\{ \norm{f_i}_{U^{s+1}(X,\mu,T)} : i\in [s]\} ,$$
where the last inequality comes from \cite[Lemma 6.7, part (i)]{kmrr25}, concluding the result.
\end{proof}

Now we prove the following theorem which extends \cref{W-erg-avera} and \cite[Theorem 6.6]{kmrr25}.

\begin{prop}\label{general-version}
    Let $(X,\mu,T)$ be ergodic, let $s\in N$, let $(\Phi_N)_{N\in \N}$ be a F\o lner sequence, and $\tau\in \mathcal{M}(X^{s})$ be invariant with respect to the transformation $T^{l_1}\times T^{l_2}\times \cdots \times T^{l_{s}}$ for $l_1,\ldots,l_s\in \N_0$. Assume that the marginals $\tau_i$ of $\tau$ satisfy $\tau_i\leq c_i \mu$ for some constant $c_i>0$ for all $1\leq i \leq s$. Then, for any $f_1,\ldots,f_{s}\in L^\infty(\mu)$ with $\norm{f_i}_{L^\infty(\mu)}\leq 1$, any bounded sequence $b:\N\to \C$, and any $m_1,\ldots,m_s\in \N$ we have
    \begin{align*}
            \limsup_{N\to\infty} \norm{\bE_{n\in \Phi_N} b(n) \bigotimes_{i=1}^s T^{m_in}   f_{i} 
   }_{L^2(\sigma)}\leq C \norm{b}_\infty \min\{ \norm{f_i}_{U^{s+1}(X,\mu,T)} :  i\in [s]\} .
    \end{align*}
    where $C$ is a constant depending on $m_1,\ldots,m_s,c_1,\ldots,c_s,l_1,\ldots,l_s$ and $s$.
\end{prop}
\begin{proof}
Let $M=LCM(l_1,\ldots,l_s)$. By \cite[Lemma 6.8]{kmrr25} we have that 
$$ \limsup_{N\to\infty} \norm{\bE_{n\in \Phi_N} b(n) \bigotimes_{i=1}^s T^{m_in}   f_{i} 
   }_{L^2(\sigma)} \leq \frac{1}{M} \sum_{j=0}^{M-1}\limsup_{N\to\infty} \norm{ \frac{1}{|\Psi_N|}  \sum_{n\in \Psi_N}  b(Mn+j) \bigotimes_{i=1}^s T^{m_i(Mn+j)}   f_{i}}_{L^2(\sigma)} $$
where $\Psi_N:=\Phi_N/M$. Denote $S=T^{l_1}\times \cdots \times T^{l_{s}}$. For $j\in \{0,\ldots,M-1\}$ we have that 
\begin{align*}
  & \limsup_{N\to\infty} \norm{ \frac{1}{|\Psi_N|}  \sum_{n\in \Psi_N}  b(Mn+j) \bigotimes_{i=1}^s T^{m_i(Mn+j)}   f_{i}}_{L^2(\sigma)} \\
   &=  \limsup_{N\to\infty} \norm{ \frac{1}{|\Psi_N|}  \sum_{n\in \Psi_N}  b(Mn+j) \prod_{i=1}^s S^{\frac{Mm_in}{l_i}}( 1^{[i-1]}\otimes   T^{m_ij}f_{i} \otimes 1^{[s-i]} )}_{L^2(\sigma)}\\
   &\leq C' \norm{b}_\infty \min\{   \norm{  1^{[i-1]}\otimes   T^{m_ij}f_{i} \otimes 1^{[s-i]} }_{U^{s+1}(X^s,\tau,S)} : i\in [s] \}\\
   &=  C' \norm{b}_\infty \min\{   \norm{  T^{m_ij}f_{i}}_{U^{s+1}(X,\tau_i,T^{l_i})} : i\in[s]\}.  
\end{align*}
where we used \cref{W-erg-avera}. Observe that by \cite[Lemma 6.7, part (ii)]{kmrr25}
$$ \norm{  T^{m_ij}f_{i}}_{U^{s+1}(X,\tau_i,T^{l_i})}\leq c_i l_i^{\frac{(s+1)}{2^{s+1}}}\norm{  T^{m_ij}f_{i}}_{U^{s+1}(X,\mu,T)} =c_i l_i^{\frac{(s+1)}{2^{s+1}}}\norm{ f_{i}}_{U^{s+1}(X,\mu,T)}.  $$
Since this bound does not depend on $j=0,\ldots,M-1$, we conclude that 
$$\limsup_{N\to\infty} \norm{ \frac{1}{|\Psi_N|}  \sum_{n\in \Psi_N}  b(Mn+j) \bigotimes_{i=1}^s T^{m_i(Mn+j)}   f_{i}}_{L^2(\sigma)}\leq C \norm{b}_\infty \min\{  \norm{ f_{i}}_{U^{s+1}(X,\mu,T)} : i\in[s]\}. $$
where $C:= C'\cdot \max\{ c_i l_i^{\frac{(s+1)}{2^{s+1}}}:i\in [s]\}$, concluding the result. 
\end{proof}
The way we use \cref{general-version} is condensed in the following corollary.
\begin{cor}\label{first-approx}
    There is a $C>0$ such that for every bounded sequence $b:\N\to \C$ and $f_w\in L^\infty(X)$ with $\norm{f_w}_{L^\infty(\mu)}\leq 1$ for each $w\in \mathcal{W}_{d}^{+}$, we have  
    \begin{align*}
            \limsup_{N\to\infty} \norm{\bE_{n\in \Phi_N} b(n) \bigotimes_{w\in \mathcal{W}_{d}\setminus\sA^d }  \prod_{\substack{j\in\sA\\ wj\notin \sF}} T^{jn}   f_{wj} \otimes 1^{\sA^d\setminus \sF}
   }_{L^2(\sigma)}\leq C \norm{b}_\infty \min\{ \norm{f_w}_{U^L} :  w\in \mathcal{W}_{d}^{+}\}.
    \end{align*}
\end{cor}
\begin{proof}
    First, notice that the marginals of $\sigma$ are such that for each $w\in \mathcal{W}_d$, $\sigma_w\leq c_w \mu$ for some $c_w>0$. Consider the measure $\nu$ on $X^{|\mathcal{W}_{d}^{+}|}$ defined by the relation 
    $$\int_{X^{|\mathcal{W}_d|}} \bigotimes_{w\in \mathcal{W}_{d}\setminus\sA^d }  \prod_{\substack{j\in\sA\\ wj\notin \sF}}  f_{wj} \otimes 1^{\sA^d\setminus \sF} d\sigma = \int_{X^{| \mathcal{W}_{d}^{+}|} } \bigotimes_{w\in \mathcal{W}_d^{+}} f_w d\nu. $$
    Notice that $\nu$ is a $(T^{k+ \sum_{i=1}^{|w^-|}w_i} )_{w\in \mathcal{W}_d^{+}}$ invariant measure. In addition, for each $w\in \mathcal{W}_{d}^{+}$, $\nu_{w}= \sigma_{w^- }\leq c_{w^- }\cdot \mu. $ By \cref{general-version} we conclude that there is $C>0$ satisfying
    $$          \limsup_{N\to\infty} \norm{\bE_{n\in \Phi_N} b(n) \bigotimes_{w\in \mathcal{W}_d^{+}}  T^{w_{|w|}n}f_w  
   }_{L^2(\nu)}\leq C \norm{b}_\infty \min\{ \norm{f_i}_{U^{L}(X,\mu,T)} : w\in \mathcal{W}_d^{+}\} . $$
   Observing that 
   $$\norm{\bE_{n\in \Phi_N} b(n) \bigotimes_{w\in \mathcal{W}_d^{+}} T^{w_{|w|}n}f_w  
   }_{L^2(\nu)}=  \norm{\bE_{n\in \Phi_N} b(n) \bigotimes_{w\in \mathcal{W}_{d}\setminus\sA^d }  \prod_{\substack{j\in\sA\\ wj\notin \sF}} T^{jn}   f_{wj} \otimes 1^{\sA^d\setminus \sF}
   }_{L^2(\sigma)} $$
   yields the conclusion.
\end{proof}

We will borrow the following lemma from \cite{kmrr25} which relates the ergodic averages of a continuous function correlated with a nilsequence.

 \begin{lemma}[{\cite[Lemma 6.11]{kmrr25}}]\label{second-approx}
 Let $(X, \mu, T)$ be an ergodic system, let $x \in \operatorname{gen}(\mu, \Phi)$ for some Følner sequence $\Phi$, and denote by $\left(\mathcal{Z}_{s-1}, \mathfrak{m}_{s-1}, T\right)$ the $(s-1)$-step pronilfactor of $(X, \mu, T)$. Assume $(Y, S)$ is a $(s-1)$-step pronilsystem. Then for every $g \in C(X), y \in Y$, and $F \in C(Y)$,
$$
\limsup _{N \rightarrow \infty}\left|\frac{1}{\left|\Phi_N\right|} \sum_{n \in \Phi_N} g\left(T^n x\right) F\left(S^n y\right)\right| \leqslant\left\|\mathbb{E}\left(g \mid \mathcal{Z}_{s-1}\right)\right\|_{L^1\left(\mathfrak{m}_{s-1}\right)} \cdot\|F\|_{\infty}.
$$
 \end{lemma}

To finish with the preliminaries, we will prove that for continuous functions the expression in \cref{Lemma-new-method} is not only positive in the pronilfactor, but also is uniformly positive.

\begin{prop}\label{Theo-in-pronil}
    For every $\eta>0$ there exists $\delta>0$ satisfying that if for each $w\in \mathcal{W}_d$, $g_w\in C(Z)$ is a non-negative function bounded by $1$, then
    \begin{equation*}
        \int \bigotimes_{w\in \mathcal{W}_d}g_w d\tilde{\sigma}>\eta \Rightarrow   \uclim_{n}\int      \Big( \tilde{T}^n \bigotimes_{w\in \mathcal{W}_d} g_{w}\Big)\left(\pi(a), (x_{w^-})_{w\in \mathcal{W}_{d}^{+}} \right)\cdot \Big(\bigotimes_{w\in \mathcal{W}_d}g_{w}  \Big)(x)  d\tilde{\sigma}(x)    >\delta.
    \end{equation*}
\end{prop}
\begin{proof}
Denote $\mathfrak{a}:=\pi(a)$. Let $(g_w)_{w\in \mathcal{W}_d}\subseteq C(Z)$ be non-negative functions bounded by $1$ such that
\begin{equation}\label{eq-25} 
      \uclim_{n\to \infty} \left(\vec{T}^n\bigotimes_{w\in \mathcal{W}_d}g_w \right)(\mathfrak{a}^{\mathcal{W}_d})=     \int \bigotimes_{w\in \mathcal{W}_d}g_w d\tilde{\sigma}>\eta .
 \end{equation}

Define $$P=\left\{ n\in \N \mid \prod_{w\in \mathcal{W}_d}g_w(T^{(k+\sum_{i=1}^{|w|}w_i)n}\mathfrak{a})>\eta/2 \right\}. $$

By \cref{eq-25} we have that $P$ has lower Banach density greater than $\eta/2$. Let $R=r\cdot LCM(\{|k+\sum_{i=1}^{|w|}w_i |\}_{w\in \mathcal{W}_d})$. By Szemerédi's theorem (\cref{Uniform-Szemeredi}), we have that there is a constant $c>0$ only depending on $\eta$ such that
\begin{equation}\label{eq-26}
     \overline{d}\Big(\left\{m\in \N \mid  \overline{d}\Big((P+Rm)\cap\cdots (P+m)\cap P\cap  (P-m)\cap \cdots\cap (P-Rm)\Big)>c \right\}\Big) >c.
\end{equation}
Using the uniquely ergodicity of the orbit of $\mathfrak{a}^{\mathcal{W}_d}$ through the transformations $\Vec{T}$ and $\tilde{T}$ given by \cite[Theorem 17, Chapter 11]{Host_Kra_nilpotent_structures_ergodic_theory:2018} we get 
\begin{align*}
      & \uclim_{m}\int     \tilde{T}^m \Big( \bigotimes_{w\in \mathcal{W}_d} g_{w}\Big)(\mathfrak{a},  (x_{w^-})_{w\in \mathcal{W}_{d}^{+}})\cdot \Big(\bigotimes_{w\in \mathcal{W}_d}g_{w}  \Big)(x)    d\tilde{\sigma}(x) \\
&= \uclim_{n\to \infty} \uclim_{m\to \infty}  \Big(\bigotimes_{w\in \mathcal{W}_d}g_w\Big)\left(\tilde{T}^m \left(\mathfrak{a}, ( T^{n(k+\sum_{i=1}^{|w^-|}w_i)} \mathfrak{a})_{w\in \mathcal{W}_{d}^{+}}\right)\right)\cdot \Big(\bigotimes_{w\in \mathcal{W}_d}g_w\Big)(\vec{T}^n\mathfrak{a}^{\mathcal{W}_d} )\\
&= \uclim_{n\to \infty} \uclim_{m\to \infty} 
             \Big(\bigotimes_{w\in \mathcal{W}_d}g_w\Big)(\tilde{T}^m\Vec{T}^n \mathfrak{a}^{\mathcal{W}_d} )\cdot  \Big(\bigotimes_{w\in \mathcal{W}_d}g_w\Big)(\vec{T}^n\mathfrak{a}^{\mathcal{W}_d} ) .
\end{align*}
The last expression is bounded by below by
$$ \frac{1}{R} \uclim_{n\to \infty} \uclim_{m\to \infty} 
             \Big(\bigotimes_{w\in \mathcal{W}_d}g_w\Big)(\tilde{T}^{Rm}\Vec{T}^n \mathfrak{a}^{\mathcal{W}_d} )\cdot  \Big(\bigotimes_{w\in \mathcal{W}_d}g_w\Big)(\vec{T}^n\mathfrak{a}^{\mathcal{W}_d} ),$$
so it is enough to bound uniformly the expression
$$\frac{1}{R} \uclim_{n\to \infty} \uclim_{m\to \infty} 
           g_\varepsilon(T^{kn+kRm}\mathfrak{a})\cdot  \prod_{w\in \mathcal{W}_d^{+}} g_{w}(T^{(k+\sum_{i=1}^{|w|}w_i)n+w_{|w|}Rm}\mathfrak{a}) \cdot \prod_{w\in \mathcal{W}_d}g_w(T^{(k+\sum_{i=1}^{|w|}w_i)n}\mathfrak{a}).$$

For this, notice that for $n,m\in \N$ with $n\in (P+Rm)\cap\cdots (P+m)\cap P\cap  (P-m)\cap \cdots\cap (P-Rm)$, we have that $n+im\in P$ for each $i\in \{-R,...,R\}$. In particular, we have that $n+Rm\in P$ which implies that $g_\varepsilon(T^{kn+kRm}\mathfrak{a})>\eta/2$. Similarly, for $w\in\mathcal{W}_{d}^{+}$, we have that $n+\frac{w_{|w|}R}{(k+\sum_{i=1}^{|w|}w_i)}m\in P$, which implies that 

$$g_{w}(T^{(k+\sum_{i=1}^{|w|}w_i)(n+\frac{w_{|w|}R}{k+\sum_{i=1}^{|w|}w_i}m) }\mathfrak{a})= g_{wj}(T^{(k+\sum_{i=1}^{|w|}w_i)n+w_{|w|}Rm}\mathfrak{a})>\eta/2.$$

In addition, given that $n\in P$, we have that $g_j(T^{(k+j)n}a)>\eta/2$ for each $j\in \{-R,\ldots,R\}$. Consequently, we have that for such $n,m\in \N$, the following holds
\begin{equation}\label{eq-28}
       g_\varepsilon(T^{kn+kRm}\mathfrak{a})\cdot\prod_{w\in \mathcal{W}_d^{+}} g_{w}(T^{(k+\sum_{i=1}^{|w|}w_i)n+w_{|w|}Rm}\mathfrak{a})\cdot \prod_{w\in \mathcal{W}_d}g_w(T^{(k+\sum_{i=1}^{|w|}w_i)n}\mathfrak{a})>(\eta/2)^{2|\mathcal{W}_d|}.
\end{equation}
Therefore, by \cref{eq-26} and \cref{eq-28} allow to conclude that 
\begin{equation*}
     \uclim_{m}\int    \tilde{T}^m \Big(  \bigotimes_{w\in \mathcal{W}_d} g_{w}\Big)(\mathfrak{a}, (x_{w^-})_{w\in \mathcal{W}_{d}^{+}} )\cdot \Big(\bigotimes_{w\in \mathcal{W}_d}g_{w}  \Big)(x)   d\tilde{\sigma}(x) >(\eta/2)^{2|\mathcal{W}_d|}c^2R^{-1},
\end{equation*}
concluding. 
\end{proof}

Now we finally prove \cref{Lemma-new-method}.
\begin{proof}[Proof of \cref{Lemma-new-method}]
Let $(f_w)_{w\in \mathcal{W}_d}\subseteq C(X)$ be non-negative continuous functions bounded by $1$ without loss of generality, such that 
\begin{equation*}
         \eta':=\int  \bigotimes_{w\in\mathcal{W}_d} f_{w} d \sigma >0.
\end{equation*}
Pick $\eta\in (0,\eta'/2)  $ and notice that by definition of $\sigma$
\begin{equation*}
         \int \bigotimes_{w\in\mathcal{W}_d} \bE(f_{w}|Z)d\tilde{\sigma}= \int  \bigotimes_{w\in\mathcal{W}_d} f_{w} d\sigma>2\eta.
\end{equation*}
Let $0<\delta<\eta$ given by \cref{Theo-in-pronil} for $\eta$. We take $(g_w)_{w\in \mathcal{W}_d}\subseteq C(Z)$ to be non-negative functions bounded by $1$ such that $\norm{g_w-\E(f_w|Z)}_{L^2(Z)}\leq \delta/(2\cdot |\mathcal{W}_d|\cdot k)$ for each $w\in \mathcal{W}_d$ and thus
\begin{equation*}
       \int \bigotimes_{w\in\mathcal{W}_d} g_{w}d\tilde{\sigma}>\eta.
\end{equation*}

Hence, by \cref{Theo-in-pronil} there is $\delta>0$ such that 
\begin{equation}\label{uniform-bound-uclim}
     \uclim_{n}\int    \Big( \tilde{T}^n  \bigotimes_{w\in \mathcal{W}_d} g_{w}\Big)(\pi(a), (x_{w^-})_{w\in \mathcal{W}_d^{+}} )\cdot \Big(\bigotimes_{w\in \mathcal{W}_d}g_{w}  \Big)(x) d\tilde{\sigma}(x)  >\delta.
\end{equation}
Using \cref{first-approx} and dominated convergence theorem we have that 
\begin{align*}
    & \liminf_{N\to\infty}  \bE_{n\in \Phi_N}  T^{kn} f_\varepsilon(a) \int  \Big( (\bigotimes_{w\in \mathcal{W}_{d}\setminus\sA^d}  \prod_{\substack{j\in\sA\\wj\notin \sF}} T^{jn}f_{wj} )\otimes 1^{\sA^d\setminus \sF} \Big) \cdot  \bigotimes_{w\in \mathcal{W}_d}f_{w}d\sigma \\
    & =\liminf_{N\to\infty}  \bE_{n\in \Phi_N}  T^{kn} f_\varepsilon(a) \int  \Big( (\bigotimes_{w\in \mathcal{W}_{d}\setminus\sA^d}  \prod_{\substack{j\in\sA\\ wj\notin \sF}}T^{jn}\bE(f_{wj}|Z) )\otimes 1^{\sA^d\setminus \sF} \Big) \cdot \bigotimes_{w\in \mathcal{W}_d}f_{w}d\sigma \\
    & =\liminf_{N\to\infty}  \bE_{n\in \Phi_N}  T^{kn} f_\varepsilon(a) \int  \Big( (\bigotimes_{w\in \mathcal{W}_{d}\setminus\sA^d}  \prod_{\substack{j\in\sA\\ wj\notin \sF}}T^{jn}\bE(f_{wj}|Z) )\otimes 1^{\sA^d\setminus \sF} \Big) \cdot \bigotimes_{w\in \mathcal{W}_d}\bE(f_{w}|Z)d\tilde{\sigma} \\
    &\geq - \delta/2 + \liminf_{N\to\infty}  \bE_{n\in \Phi_N}  T^{kn} f_\varepsilon(a) \int  \Big( (\bigotimes_{w\in \mathcal{W}_{d}\setminus\sA^d}  \prod_{\substack{j\in\sA\\ wj\notin \sF}}T^{jn}g_{wj} )\otimes 1^{\sA^d\setminus \sF} \Big) \cdot \bigotimes_{w\in \mathcal{W}_d}g_{w}d\tilde{\sigma}
\end{align*}
where we used the definition of $\sigma$ on the third equality and the definition of $(g_w)_{w\in \mathcal{W}_d}$ on the last inequality. 

On the other hand, by the fact that $a\in gen_{T^k}(\mu_k,\Phi)$, we can use \cref{second-approx} with $(X,\mu_k,T^k)$, Fatou's lemma, and \cref{ergodic-decom-Tk} to get that  
\begin{align*}
    & \limsup_{N\to \infty} \Big|  \bE_{n\in \Phi_N}  T^{kn} (f_\varepsilon(a)-g_\varepsilon(\pi(a)))  \int  \Big( (\bigotimes_{w\in \mathcal{W}_{d}\setminus\sA^d}  \prod_{\substack{j\in \sA\\ wj\notin\sF}}T^{jn}g_{wj} )\otimes 1^{\sA^d\setminus \sF} \Big) \cdot  \bigotimes_{w\in \mathcal{W}_d}g_{w}d\tilde{\sigma}   \Big|\\
    &\leq \int  \limsup_{N\to \infty} \Big|  \bE_{n\in \Phi_N}  T^{kn} (f_\varepsilon(a)-g_\varepsilon(\pi(a)))  \Big( (\bigotimes_{w\in \mathcal{W}_{d}\setminus\sA^d}  \prod_{\substack{j\in \sA\\wj\notin \sF}}T^{jn}g_{wj} )\otimes 1^{\sA^d\setminus \sF} \Big)\Big| d\tilde{\sigma}    \\
    &\leq  \norm{\E(f_\varepsilon-g_\varepsilon\circ \pi | Z_k) }_{L^1(m_k)} \leq \ell\norm{\E(f_\varepsilon| Z)-g_\varepsilon }_{L^1(m)}  \leq  \ell\norm{\E(f_\varepsilon| Z)-g_\varepsilon }_{L^1(m)} \leq \delta/2
\end{align*}
in where in the second to last inequality we used that all the functions involved are $\mu$-a.e. bounded by $1$ and the fact that $g_\varepsilon$ is $Z$-measurable, and in the last inequality we used that $\ell\leq k$. By \cref{uniform-bound-uclim} we deduce that
\begin{align*}
    & \liminf_{N\to \infty}  \bE_{n\in \Phi_N}  \int    \Big(  \tilde{T}^n \bigotimes_{w\in \mathcal{W}_d} f_{w}\Big)(a, (x_{w^-})_{w\in \mathcal{W}_d^{+}} )\cdot \Big(\bigotimes_{w\in \mathcal{W}_d}f_{w}  \Big)(x) d\tilde{\sigma}(x)  \\
    &\geq -\delta+  \uclim_{N\to\infty}   \int     \Big(  \tilde{T}^n\bigotimes_{w\in \mathcal{W}_d} g_{w}\Big)(\pi(a), (x_{w^-})_{w\in \mathcal{W}_d^{+}} )\cdot \Big(\bigotimes_{w\in \mathcal{W}_d}g_{w}  \Big)(x) d\tilde{\sigma}(x)  >0,
\end{align*}
concluding the proof.
\end{proof}

\small{
\bibliographystyle{aomalpha}
\bibliography{refs}
}

\bigskip
\noindent
Felipe Hernández\\
\textsc{{\'E}cole Polytechnique F{\'e}d{\'e}rale de Lausanne} (EPFL)\par\nopagebreak
\noindent
\href{mailto:felipe.hernandezcastro@epfl.ch}
{\texttt{felipe.hernandezcastro@epfl.ch}}

\end{document}